\documentclass[a4paper,12pt]{article}

\usepackage{amsfonts}
\usepackage{amsmath}
\usepackage{amssymb}
\usepackage{latexsym}
\usepackage{indentfirst}
\usepackage{graphicx}
\usepackage{epsfig}
\usepackage{subfigure}
\usepackage{multirow}
\usepackage{tabularx}
\usepackage{placeins} 
\usepackage[mathscr]{eucal}
\usepackage{color}

\newcommand{\R}{{\if mm {\rm I}\mkern -3mu{\rm R}\else \leavevmode
\hbox{I}\kern -.17em\hbox{R} \fi}}

\newcommand{\cF}{\mbox{{${\cal F}$}}}

\newcommand{\bu}{\mbox{\boldmath{$u$}}}
\newcommand{\bv}{\mbox{\boldmath{$v$}}}

\newcommand{\bx}{\mbox{\boldmath{$x$}}}

\newcommand{\fb}{\mbox{\boldmath{$f$}}}

\newcommand{\bxi}{\mbox{\boldmath{$\xi$}}}

\newcommand{\bsigma}{\mbox{\boldmath{$\sigma$}}}
\newcommand{\btau}{\mbox{\boldmath{$\tau$}}}
\newcommand{\bvarepsilon}{\mbox{\boldmath{$\varepsilon$}}}
\newcommand{\bnu}{\mbox{\boldmath{$\nu$}}}

\newcommand{\bzero}{\mbox{\boldmath{$0$}}}

\newcommand{\weak}{\rightharpoonup}

\def\bar{\overline}
\def\real{\mathbb{R}}

\newtheorem{theorem}{Theorem}[section]
\newtheorem{lemma}[theorem]{Lemma}

\newtheorem{corollary}[theorem]{Corollary}

\newtheorem{proposition}[theorem]{Proposition}

\numberwithin{equation}{section}
\newenvironment{proof}[1][Proof]{\textbf{#1.} }
{\ \rule{0.75em}{0.75em}\smallskip}

\begin{document}

\begin{center}
\Large\bf Convergence Analysis of Penalty Based Numerical Methods for Constrained Inequality Problems
\end{center}

\begin{center}
Weimin Han\footnote{Program in Applied Mathematical and Computational Sciences (AMCS)
\& Department of Mathematics, University of Iowa, Iowa City, IA 52242, USA, 
e-mail: {\tt weimin-han@uiowa.edu}.  The work was supported by NSF grant DMS-1521684.}
\quad and\quad 
Mircea Sofonea\footnote{Laboratoire de Math\'ematiques et Physique.
Universit\'e de Perpignan Via Domitia, 52 Avenue Paul Alduy, 66860 Perpignan, France;
e-mail: {\tt sofonea@univ-perp.fr}}
\end{center}

\medskip

\begin{quote}
\noindent {\bf Abstract.} \ 
This paper presents a general convergence theory of penalty based numerical methods for 
elliptic constrained inequality problems, including variational inequalities, hemivariational 
inequalities, and variational-hemivariational inequalities.  The constraint is relaxed by a 
penalty formulation and is re-stored as the penalty parameter tends to zero.  The main theoretical
result of the paper is the convergence of the penalty based numerical solutions to the solution
of the constrained inequality problem as the mesh-size and the penalty parameter approach zero
simultaneously but independently.  The convergence of the penalty based numerical methods is
first established for a general elliptic variational-hemivariational inequality with constraints,
and then for hemivariational inequalities and variational inequalities as special cases.
Applications to problems in contact mechanics are described.

\smallskip
\noindent {\bf Keywords.} \ 
Hemivariational inequality, variational inequality,  variational-hemivariational inequality, 
penalty based numerical methods, convergence, contact mechanics

\smallskip
\noindent {\bf 2010 Mathematics Subject Classification.} \ 
65N30, 65N15, 74M10, 74M15
\end{quote}

\section{Introduction}\label{s1}

Penalty methods are an effective approach in the numerical solution of problems with constraints.
In a penalty method, the constraint of the original problem is allowed to be violated but the
violation is penalized.  The constraint is restored in the limit when a small penalty parameter
goes to zero.  Penalty methods have been applied to a variety of constrained problems of importance,
e.g., penalty methods for the incompressibility constraint in incompressible fluid flow
problems (e.g., \cite{GR1986}), in contact problems (e.g., \cite{KO1988}), or in the context of 
general variational inequalities (e.g., \cite{GLT1981}).  In most related
references, convergence of the penalty based numerical method is carried out either at the 
continuous level or for an arbitrary but fixed finite dimensional approximation of the original
constrained problem.  More precisely, denote by $\epsilon>0$ the small penalty parameter,
by $h>0$ the mesh-size for the finite dimensional approximation; and let $u$, $u_\epsilon$,
$u^h$, and $u^h_\epsilon$ be the solution of the original constrained problem, the solution of the penalized
problem at the continuous level, the numerical solution of the original constrained problem, 
and the numerical solution of the penalized problem, respectively.  Then a typical convergence result found
in the literature for the penalty methods is of the type $u_\epsilon\to u$ as $\epsilon\to 0$, 
or for $h$ fixed, $u_\epsilon^h\to u^h$ as $\epsilon\to 0$.  In this paper, we will establish
the convergence result $u_\epsilon^h\to u$ as $h,\epsilon\to 0$ and we will achieve
this for families of constrained inequality problems, including variational inequalities and
hemivariational inequalities.

Hemivariational inequalities were introduced in early 1980s by Panagiotopoulos in the context of 
applications in engineering problems involving non-monotone and possibly multi-valued 
constitutive or interface laws for deformable bodies.  Studies of hemivariational inequalities can be 
found in several comprehensive references, e.g., \cite{NP1995, P1993}, and more recently, \cite{MOS2013}.
The book \cite{HMP1999} is devoted to the finite element approximation of 
hemivariational inequalities, where convergence of the numerical methods is discussed.
In the recent years, there have been efforts to derive error estimates.   In the literature, 
the paper \cite{HMS14} provides the first optimal order error estimate for the linear finite element method 
in solving hemivariational or variational-hemivariational inequalities.  The idea of the derivation 
technique in \cite{HMS14} was adopted in several papers by various authors for deriving optimal order 
error estimates for the linear finite element method of a few individual hemivariational or 
variational-hemivariational inequalities. More recently, we have developed general frameworks of 
convergence theory and error estimation for hemivariational or variational-hemivariational inequalities: 
for internal numerical approximations of general hemivariational and variational-hemivariational 
inequalities in \cite{HSB, HSD}, and for both internal and external numerical approximations of 
general hemivariational and variational-hemivariational inequalities in \cite{Han17}.  
In these recent papers, convergence is shown for numerical solutions by internal or external approximation 
schemes under minimal solution regularity condition, C\'{e}a type inequalities are derived that serve as 
the starting point for error estimation, for hemivariational and variational-hemivariational inequalities 
arising in contact mechanics, optimal order error estimates for the linear finite element solutions are derived.  

In \cite{MOS30}, well-posedness of a family of variational-hemivariational inequalities was established.
In addition, a penalty formulation for the constrained variational-hemivariational inequalities was introduced 
and convergence of the penalty solutions is shown when the penalty parameter tends to zero.  
The penalty method has also been used in the study of history-dependent variational or 
variational-hemivariational inequalities, in \cite{SP14} and \cite{SMH18}, respectively, 
where convergence of the penalty method is shown as the penalty parameter goes to zero.
In \cite{HMS18}, a penalty based numerical method is introduced and studied for sample hemivariational 
inequalities in unilateral contact mechanics.  In this paper, we substantially extend the result 
in \cite{HMS18} to cover penalty based numerical methods for solving general constrained inequality problems,
including variational inequalities, hemivariational inequalities and variational-hemivariational inequalities,
prove convergence of the numerical solutions as both the meshsize and the penalty parameter tend to zero.

The rest of the paper is organized as follows.  In Section~\ref{s2} we review some basic notions 
needed in the study of variational-hemivariational inequalities.
In Section~\ref{sec:VHVI}, we introduce a general variational-hemivariational inequality. 
In Section~\ref{s3} we describe numerical methods based on penalty formulation for solving the constrained 
variational-hemivariational inequalities, and show convergence of the penalty based numerical methods 
in Section \ref{sec:conver}.  Results on variational-hemivariational inequalities automatically 
reduce to corresponding ones on purely hemivariational inequalities and purely variational inequalities, 
respectively, with simplified conditions. We  include some comments in Section \ref{sec:particular} 
on the convergence of the penalty based numerical methods for such inequalities. Finally,
in Section \ref{s4} we illustrate the application of the results from previous sections in the 
study of the penalty based numerical methods for two mathematical models describing
unilateral contact of an elastic body with an obstacle.

\section{Preliminaries}\label{s2}

We introduce some basic notions and results in this section.  All the spaces used in this paper are real.  
For a normed space $X$, we denote by $\|\cdot\|_X$ its norm, by $X^*$ its dual space, and by 
$\langle\cdot,\cdot\rangle_{X^*\times X}$ the duality pairing of $X^*$ and $X$. 
When no confusion may arise, we simply write 
$\langle\cdot,\cdot\rangle$ instead of $\langle\cdot,\cdot\rangle_{X^*\times X}$. 
Weak convergence is indicated by the symbol $\weak$.  The space of all linear continuous operators 
from one normed space $X$ to another normed space $Y$ is denoted by ${\cal L}(X,Y)$.

An operator  $A \colon X \to X^*$ is said to be pseudomonotone if it is bounded and $u_n \weak u$ in $X$
together with $\displaystyle \limsup_{n\to\infty}\,\langle A u_n, u_n -u \rangle_{X^*\times X} \le 0$ imply 
\[ \langle A u, u - v \rangle_{X^*\times X} \le 
   \liminf_{n\to\infty}\,\langle A u_n, u_n - v \rangle_{X^*\times X}\quad\forall\,v \in X. \]
The operator $A$ is said to be demicontinuous if $u_n\to u$ in $X$ implies $Au_n\weak Au$ in $X^*$. 
A function $\varphi \colon X \to \real$ is said to be lower semicontinuous (l.s.c.)
if $x_n \to x$ in $X$ implies $\varphi (x) \le \liminf_{n\to\infty} \varphi (x_n)$.  
For a convex function $\varphi$,   the set
\[ \partial_c\varphi (x) := \left\{ x^*\in X^* \mid
   \varphi(v)-\varphi(x)\ge \langle x^*, v -x \rangle_{X^*\times X}\ \forall\, v \in X\right\} \]
is known as the subdifferential (in the sense of convex analysis) of $\varphi$  at $x\in X$.  Elements 
in $\partial\varphi(x)$ are called subgradients of $\varphi$ at $x$.  Properties of convex functions 
can be found in \cite{ET}.  A continuity result on convex functions is as follows (cf.\ \cite[p.\ 13]{ET}).

\begin{lemma}\label{lem:continuity}
A l.s.c.\ convex function $\varphi:X\to\real$ on a Banach space $X$ is continuous. 
\end{lemma}

Another result on convex functions that we will use is a lower bound of a convex function 
by a continuous affine functional (cf.\ \cite[Lemma 11.3.5]{AH2009}, \cite[Prop.\ 5.2.25]{DMP1}).

\begin{lemma}\label{lem:lower}
Let $Z$ be a normed space and let $\varphi:Z\to\mathbb{R}$ be 
convex and l.s.c.  Then there exist a continuous linear functional
$\ell_\varphi\in Z^\prime$ and a constant $\overline{c}\in \mathbb{R}$ such that
\[ \varphi(z)\ge \ell_\varphi(z)+\overline{c}\quad\forall\,z\in Z.\]
Consequently, there exist two constants $\overline{c}$ and $\tilde{c}$ such that 
\begin{equation}
\varphi(z)\ge \overline{c}+\tilde{c}\,\|z\|_Z\quad\forall\,z\in Z.
\label{eq2.1}
\end{equation}
\end{lemma}

Assume $\psi \colon X \to \real$ is locally Lipschitz continuous.  The generalized (Clarke) 
directional derivative of $\psi$ at $x \in X$ in the direction $v \in X$ is defined by
\begin{equation*}
\psi^{0}(x; v) := \limsup_{y \to x, \, \epsilon \downarrow 0} 
\frac{\psi(y + \epsilon v) - \psi(y)}{\epsilon}\,.
\end{equation*}
The generalized subdifferential of $\psi$ at $x$ is a subset of the dual space $X^*$ given by
\begin{equation*}
\partial \psi (x) := \{\, \zeta \in X^* \mid \psi^{0}(x; v) \ge
{\langle \zeta, v \rangle}_{X^* \times X} \ \forall\, v \in X \, \}.
\end{equation*}
Details on properties of the 
subdifferential in the Clarke sense can be found in the books \cite{Clarke,DMP1,MOS2013,NP1995}.

On several occasions, we will apply the modified Cauchy inequality: for any $\delta>0$,
there exists a constant $c$ depending only on $\delta$ such that
\begin{equation}
a\,b\le \delta\,a^2+c\,b^2\quad\forall\,a,b\in\mathbb{R}.
\label{eq2.2}
\end{equation}
In fact, we may simply take $c=1/(4\,\delta)$ in \eqref{eq2.2}.

\section{A general constrained variational-hemivariational inequality}\label{sec:VHVI}

The constrained variational-hemivariational inequality problem was studied in \cite{MOS30}.  
Here, we follow the presentation in \cite{Han17} to describe the problem.
Let $X, X_\varphi, X_j$ be normed spaces and $K \subsetneq X$.  Let there be given operators 
$A \colon X \to X^*$, $\gamma_\varphi:X\to X_\varphi$, $\gamma_j:X\to X_j$, and functionals 
$\varphi \colon  X_\varphi\times  X_\varphi\to \real$, $j \colon X_j \to \real$, $j$ being
locally Lipschitz.  The variational-hemivariational inequality we will consider is stated as follows.
\smallskip

\noindent {\sc Problem} (P). 
{\it Find an element  $u \in K$ such that}
\begin{align}
& \langle Au,v-u\rangle+\varphi(\gamma_\varphi u,\gamma_\varphi v)-\varphi(\gamma_\varphi u,\gamma_\varphi u)
\nonumber\\
&\qquad{} + j^0(\gamma_j u;\gamma_j v -\gamma_j u)\ge \langle f, v - u \rangle\quad\forall\, v \in K.
\label{hv}
\end{align}

We will make use of the following conditions.

\smallskip

$(A_1)$ $X$ is a reflexive Banach space, $K\subsetneq X$ is non-empty, closed and convex.

$(A_2)$ $X_\varphi$ is a Banach space and $\gamma_\varphi\in{\cal L}(X,X_\varphi)$ with a
continuity constant $c_\varphi>0$:
\begin{equation}
\|\gamma_\varphi v\|_{X_\varphi}\le c_\varphi \|v\|_X\quad\forall\,v\in X. 
\label{Xphi}
\end{equation}

$(A_3)$ $X_j$ is a Banach space and $\gamma_j\in{\cal L}(X,X_j)$ with a continuity constant $c_j>0$:
\begin{equation}
\|\gamma_j v\|_{X_j}\le c_j \|v\|_X\quad\forall\,v\in X. 
\label{Xj}
\end{equation}

$(A_4)$ $A \colon X \to X^*$ is pseudomonotone and strongly monotone with a 
constant $m_A > 0$:
\begin{equation}
\langle Av_1 - Av_2, v_1 - v_2 \rangle\ge m_A \| v_1 - v_2 \|_X^2\quad \forall\, v_1, v_2 \in X.
\label{assum:A}
\end{equation}

$(A_5)$ $\varphi \colon X_\varphi \times X_\varphi \to \real$ is such that
$\varphi (z,\cdot)\colon  X_\varphi\to\real$ is convex and l.s.c.\ for all $z\in X_\varphi$, 
and for a constant $\alpha_\varphi \ge 0$,
\begin{align}
& \varphi (z_1, z_4) - \varphi (z_1, z_3) + \varphi (z_2,z_3) - \varphi (z_2,z_4) \nonumber\\
&{}\qquad \le \alpha_\varphi\|z_1-z_2\|_{X_\varphi}\|z_3-z_4\|_{X_\varphi}
\quad\forall\,z_1,z_2, z_3,z_4\in X_\varphi. \label{assum:phi}
\end{align}

$(A_6)$ $j\colon X_j\to\real$ is locally Lipschitz, and for some constants $c_0,c_1,\alpha_j\ge 0$,
\begin{align}
&\| \partial j(z) \|_{X^*_j} \le c_0 + c_1 \|z\|_{X_j} \quad \forall\,z\in X_j,\label{assum:j1} \\
&j^0(z_1;z_2 -z_1) + j^0(z_2;z_1 -z_2) \le \alpha_j \|z_1 -z_2 \|_{X_j}^2 \quad\forall\,z_1,z_2\in X_j.
\label{assum:j2}
\end{align}

$(A_7)$
\begin{equation}
\alpha_\varphi c_\varphi^2 + \alpha_j c_j^2 < m_A. \label{small}
\end{equation}

$(A_8)$
\begin{equation}\label{assum:f}
f \in X^*.
\end{equation}

The inequality (\ref{hv}) represents a variational-hemivariational inequality since the 
function $\varphi(z, \cdot)$ is assumed to be convex for any $z\in X_\varphi$ and the function $j$ 
is assumed locally Lipschitz and generally nonconvex.   Note that we assume $K$ is a proper subset of $X$, 
and so the corresponding inequality \eqref{hv} is termed a constrained variational-hemivariational inequality.  
The spaces $X_\varphi$ and $X_j$ were introduced to facilitate error analysis of 
numerical solutions of Problem~(P).  This is useful also for the convergence  analysis of the 
penalty based numerical method in this paper. For applications in contact mechanics, 
the functionals $\varphi(\cdot,\cdot)$ and $j(\cdot)$ are integrals over the contact boundary $\Gamma_3$.
In such a situation, $X_\varphi$ and $X_j$ can be chosen to be $L^2(\Gamma_3)^d$ and/or $L^2(\Gamma_3)$.
Moreover, $\gamma_\varphi\colon X\to X_\varphi$ and $\gamma_j\colon X\to X_j$ are linear, 
continuous and compact operators. 

As was noted in \cite{Han17}, by slightly modifying the proof in \cite{MOS30} (see also
\cite[Remark 13]{SM2018}), the following existence and uniqueness result can be proved.

\begin{theorem}\label{t1}
Under assumptions $(A_1)$--$(A_8)$, Problem {\rm (P)} has a unique solution $u \in K$.
\end{theorem}

We keep assumptions $(A_1)$--$(A_8)$ throughout the paper so that a unique solution $u\in K$
is guaranteed for Problem (P).  

\section{Numerical approximations}\label{s3}

We now introduce a penalty based numerical method of Problem (P). We say that $P:X\to X^*$ is 
a penalty operator for the set $K$ if $P$ is bounded, demicontinuous, monotone, and ${\rm Ker}(P)=K$.
Note that a penalty operator thus defined is pseudomonotone, following \cite[Prop.\ 27.6]{ZeidIIB}.
We denote by $\epsilon>0$ the penalty parameter.  The penalty formulation of Problem (P) is as follows.

\smallskip

\noindent {\sc Problem} (P${}_\epsilon$). 
{\it Find an element  $u_\epsilon \in X$ such that}
\begin{align*}
& \langle Au_\epsilon,v-u_\epsilon\rangle+\frac{1}{\epsilon}\langle P u_\epsilon,v-u_\epsilon\rangle
+\varphi(\gamma_\varphi u_\epsilon,\gamma_\varphi v)-\varphi(\gamma_\varphi u_\epsilon,\gamma_\varphi u_\epsilon)\\
&\qquad{} + j^0(\gamma_j u_\epsilon;\gamma_j v -\gamma_j u_\epsilon)\ge \langle f, v - u_\epsilon \rangle
\quad\forall\, v \in X.
\end{align*}

For the particular case where $\varphi$ does not depend on its first argument, it is shown in
\cite{MOS30} that Problem (P${}_\epsilon$) has a unique solution $u_\epsilon \in X$ and that
$u_\epsilon \to u$ in $X$ as $\epsilon\to 0$.  In this paper, we consider numerical methods 
for solving Problem (P) based on the penalty formulation.  

Let $X^h\subset X$ be a finite dimensional subspace with $h>0$ being a spatial discretization parameter.  
In practice, $X^h$ is usually constructed as a finite element space.  We need an assumption 
on the approximability of elements of $K$ by elements in $X^h$.

\smallskip
$(A_9)$  For any $v\in X$, there exists $v^h\in X^h$ such that
\[ \lim_{h\to 0}\|v^h-v\|_X=0.\]
For any $v\in K$, there exists $v^h\in X^h\cap K$ such that
\begin{equation}
\lim_{h\to 0}\|v^h-v\|_X=0. \label{eq4}
\end{equation}

Then the numerical method for solving Problem (P) based on penalty formulation is the following.

\smallskip
\noindent {\sc Problem} (P${}^h_\epsilon$). 
{\it Find an element  $u^h_\epsilon\in X^h$ such that}
\begin{align}
&\langle Au^h_\epsilon,v^h-u^h_\epsilon\rangle+\frac{1}{\epsilon}\langle P u^h_\epsilon,v^h-u^h_\epsilon\rangle
+ \varphi(\gamma_\varphi u^h_\epsilon,\gamma_\varphi v^h)
- \varphi(\gamma_\varphi u^h_\epsilon,\gamma_\varphi u^h_\epsilon) \nonumber\\
&\qquad{} +j^0(\gamma_j u^h_\epsilon;\gamma_j v^h -\gamma_j u^h_\epsilon)\ge\langle f,v^h-u^h_\epsilon\rangle
\quad\forall\,v^h\in X^h. \label{eq2}
\end{align}

We can apply the arguments of the proof of Theorem \ref{t1} in the setting
of the finite dimensional space $X^h$ to conclude that under assumptions $(A_1)$--$(A_8)$,
Problem (P${}^h_\epsilon$) has a unique solution $u^h_\epsilon\in X^h$.  The main goal of the paper
is to show convergence of $u^h_\epsilon$ to the solution $u\in K$ of Problem (P) as $h$ and $\epsilon$ 
simultaneously and independently approach zero.   As a preparation for the convergence analysis of 
the numerical method in Section \ref{sec:conver},  we first prove a uniform boundedness property for 
the numerical solutions $\{u^h_\epsilon\}_{h,\epsilon>0}$.  Since $K$ is non-empty, we choose and 
fix an element $u_0$ from $K$.  By assumption $(A_9)$, there exists $u^h_0\in X^h\cap K$ such that
\begin{equation}
\lim_{h\to 0}\|u_0^h-u_0\|_X=0. \label{eq4u}
\end{equation}

We will need the following uniform boundedness property of the numerical solutions.

\begin{proposition}\label{prop:bd}
There exists a constant $M>0$ such that $\|u^h_\epsilon\|_X\le M$ $\forall\,h,\epsilon>0$.
\end{proposition}
\begin{proof}
By the strong monotonicity of $A$, we have 
\[ m_A\|u^h_\epsilon-u_0^h\|_X^2\le \langle A u^h_\epsilon,u^h_\epsilon-u_0^h\rangle
 -\langle A u^h_0,u^h_\epsilon-u_0^h\rangle. \]
Applying \eqref{eq2} with $v^h=u^h_0$, we further have
\begin{align}
m_A\|u^h_\epsilon-u_0^h\|_X^2 & \le \frac{1}{\epsilon}\langle P u^h_\epsilon,u^h_0-u^h_\epsilon\rangle
+ \varphi(\gamma_\varphi u^h_\epsilon,\gamma_\varphi u^h_0)
- \varphi(\gamma_\varphi u^h_\epsilon,\gamma_\varphi u^h_\epsilon) 
+j^0(\gamma_j u^h_\epsilon;\gamma_j u^h_0 -\gamma_j u^h_\epsilon) \nonumber\\
&\quad{} -\langle f,u^h_0-u^h_\epsilon\rangle-\langle A u^h_0,u^h_\epsilon-u_0^h\rangle. \label{eq5}
\end{align}

Let us bound various terms on the right side of \eqref{eq5}.  First, since $u^h_0\in K$, 
$P u^h_0=0$, and by the monotonicity of $P$, we have
\begin{equation}
\frac{1}{\epsilon}\langle P u^h_\epsilon,u^h_0-u^h_\epsilon\rangle =
-\frac{1}{\epsilon}\langle P u^h_\epsilon- P u^h_0,u^h_\epsilon-u^h_0\rangle \le 0.
\label{eq6}
\end{equation}

In \eqref{assum:phi} we take $z_1=z_3=\gamma_\varphi u^h_\epsilon$, $z_2=\gamma_\varphi u_0$
and $z_4=\gamma_\varphi u^h_0$ to get
\begin{align}
\varphi(\gamma_\varphi u^h_\epsilon,\gamma_\varphi u^h_0)
- \varphi(\gamma_\varphi u^h_\epsilon,\gamma_\varphi u^h_\epsilon)  
& \le \alpha_\varphi \|\gamma_\varphi (u^h_\epsilon-u_0)\|_{X_\varphi}
  \|\gamma_\varphi (u^h_\epsilon-u^h_0)\|_{X_\varphi}\nonumber\\
&{}\quad +\varphi(\gamma_\varphi u_0,\gamma_\varphi u^h_0)-\varphi(\gamma_\varphi u_0,\gamma_\varphi u^h_\epsilon).
\label{4.5a}
\end{align}
Now,
\[ \|\gamma_\varphi (u^h_\epsilon-u_0)\|_{X_\varphi} \|\gamma_\varphi (u^h_\epsilon-u^h_0)\|_{X_\varphi}
 \le c_\varphi^2 \left(\|u^h_\epsilon-u^h_0\|_X^2+\|u^h_0-u_0\|_X\|u^h_\epsilon-u^h_0\|_X\right).\]
By the modified Cauchy inequality, for any $\delta>0$, there is a constant $c$ depending on 
$\delta$ such that
\[ \alpha_\varphi c_\varphi^2\|u^h_0-u_0\|_X\|u^h_\epsilon-u^h_0\|_X\le \delta\, \|u^h_\epsilon-u^h_0\|_X^2
 +c\,\|u^h_0-u_0\|_X^2. \]
By $(A_{5})$, $\varphi(\gamma_\varphi u_0,\cdot)\colon X\to\real$ is convex and l.s.c.  So from 
\eqref{eq2.1}, for some constants $\overline{c}$ and $\tilde{c}$,
\[ \varphi(\gamma_\varphi u_0,z)\ge \overline{c}+\tilde{c}\,\|z\|_{X_\varphi}\quad\forall\,z\in X_\varphi, \]
and then
\[ -\varphi(\gamma_\varphi u_0,\gamma_\varphi u^h_\epsilon)
\le -\overline{c}-\tilde{c}\,\|\gamma_\varphi u^h_\epsilon\|_{X_\varphi}.\]
By the continuity from Lemma \ref{lem:continuity},
\[ \varphi(\gamma_\varphi u_0,\gamma_\varphi u^h_0)\to \varphi(\gamma_\varphi u_0,\gamma_\varphi u_0)\]
and so $\{\varphi(\gamma_\varphi u_0,\gamma_\varphi u^h_0)\}_{h>0}$ is uniformly bounded with respect to $h$.
Summarizing the above relations with standard manipulations, we have
\begin{equation}
\varphi(\gamma_\varphi u^h_\epsilon,\gamma_\varphi u^h_0)
- \varphi(\gamma_\varphi u^h_\epsilon,\gamma_\varphi u^h_\epsilon)
\le \left(\alpha_\varphi c_\varphi^2+\delta\right)\|u^h_\epsilon-u^h_0\|_X^2
+c\left(1+\|u^h_\epsilon-u^h_0\|_X\right).
\label{eq7}
\end{equation}

Write
\[ j^0(\gamma_j u^h_\epsilon;\gamma_j u^h_0 -\gamma_j u^h_\epsilon)
 =\left[ j^0(\gamma_j u^h_\epsilon;\gamma_j u^h_0 -\gamma_j u^h_\epsilon)
 +j^0(\gamma_j u^h_0;\gamma_j u^h_\epsilon-\gamma_j u^h_0)\right]
 -j^0(\gamma_j u^h_0;\gamma_j u^h_\epsilon-\gamma_j u^h_0). \]
Use the condition \eqref{assum:j2},
\[ j^0(\gamma_j u^h_\epsilon;\gamma_j u^h_0 -\gamma_j u^h_\epsilon)
 +j^0(\gamma_j u^h_0;\gamma_j u^h_\epsilon-\gamma_j u^h_0) \le \alpha_j c_j^2 \|u^h_\epsilon-u^h_0\|_X^2.\]
Use the condition \eqref{assum:j1},
\[ -j^0(\gamma_j u^h_0;\gamma_j u^h_\epsilon-\gamma_j u^h_0)\le 
 \left(c_0+c_1 c_j \|u^h_0\|_X\right) c_j \|u^h_\epsilon-u^h_0\|_X. \]
Thus,
\begin{equation}
j^0(\gamma_j u^h_\epsilon;\gamma_j u^h_0 -\gamma_j u^h_\epsilon)\le 
\alpha_j c_j^2 \|u^h_\epsilon-u^h_0\|_X^2+c\left(1+\|u^h_\epsilon-u^h_0\|_X\right).
\label{eq8}
\end{equation}

Easily,
\begin{equation}
-\langle f,u^h_0-u^h_\epsilon\rangle\le \|f\|_{X^*} \|u^h_\epsilon-u^h_0\|_X.
\label{eq9}
\end{equation}

By the assumption, $A:X\to X^*$ is pseumonotone; hence it is bounded.  Then, since $u^h_0\weak u_0$,
it follows that $\{\|Au^h_0\|_{X^*}\}_{h>0}$ is uniformly bounded with respect to $h$.  Thus,
\begin{equation}
 -\langle A u^h_0,u^h_\epsilon-u_0^h\rangle \le c\,\|u^h_\epsilon-u^h_0\|_X.
\label{eq10}
\end{equation}
 
Using \eqref{eq6}--\eqref{eq10} in \eqref{eq5}, we find that
\[ \left(m_A-\alpha_\varphi c_\varphi^2-\alpha_j c_j^2-\delta\right)\|u^h_\epsilon-u^h_0\|_X^2
 \le c\left(1+\|u^h_\epsilon-u^h_0\|_X\right). \]
Therefore, choosing $\delta>0$ sufficiently small, we find that $\{\|u^h_\epsilon-u^h_0\|_X\}_{h,\epsilon>0}$, 
and then also $\{\|u^h_\epsilon\|_X\}_{h,\epsilon>0}$, is uniformly bounded with respect to $h$ and $\epsilon$. 
\hfill
\end{proof}

We list an additional condition to be used later on:

$(A_{10})$ $\gamma_\varphi\colon X\to X_\varphi$ and $\gamma_j\colon X\to X_j$ are compact operators.

We comment that in applications to contact mechanics (cf.\ Section \ref{s4}), $\gamma_\varphi$ and 
$\gamma_j$ are trace operators from an $H^1(\Omega)$-based space to
$L^2(\Gamma_3)$-based spaces and, therefore, the assumption $(A_{10})$ is automatically valid.

\section{Convergence of the numerical method}\label{sec:conver}

We now prove the convergence of the numerical solution of Problem (P${}^h_\epsilon$) 
to that of Problem (P) as the penalty parameter $\epsilon$ and the meshsize $h$ tend to zero.

\begin{theorem}\label{thm:main}
Assume $(A_1)$--$(A_{10})$.  Then,
\begin{equation}
 u^h_\epsilon\to u\quad{\rm in\ }X\ {\rm as\ }h,\epsilon\to 0. \label{4.9a}
\end{equation}
\end{theorem}
\begin{proof}
By Proposition \ref{prop:bd},  $\{u^h_\epsilon\}_{h,\epsilon>0}$ is bounded in $X$.  Since $X$ 
is a reflexive Banach space, and the operators $\gamma_\varphi\colon X\to X_\varphi$ and 
$\gamma_j\colon X\to X_j$ are compact, we can find a sequence of $\{u^h_\epsilon\}_{h,\epsilon>0}$,
still denoted by $\{u^h_\epsilon\}$, and an element $w\in X$ such that
\begin{equation}
 u^h_\epsilon\weak w\ {\rm in}\ X,\quad \gamma_\varphi u^h_\epsilon\to \gamma_\varphi w\ {\rm in}\ X_\varphi,
 \quad \gamma_j u^h_\epsilon\to \gamma_j w\ {\rm in}\ X_j. 
\label{eq11}
\end{equation}

Let us show that $w\in K$.  By \eqref{eq2}, for any $v^h\in X^h$,
\begin{align*}
\frac{1}{\epsilon}\langle P u^h_\epsilon,u^h_\epsilon-v^h\rangle
&\le \langle Au^h_\epsilon, v^h - u^h_\epsilon\rangle + \varphi(\gamma_\varphi u^h_\epsilon,\gamma_\varphi v^h)
- \varphi(\gamma_\varphi u^h_\epsilon,\gamma_\varphi u^h_\epsilon)\\
&{}\quad +j^0(\gamma_j u^h_\epsilon;\gamma_j v^h -\gamma_j u^h_\epsilon)-\langle f,v^h-u^h_\epsilon\rangle.
\end{align*}
Similar to \eqref{4.5a}, we have
\[  \varphi(\gamma_\varphi u^h_\epsilon,\gamma_\varphi v^h)
- \varphi(\gamma_\varphi u^h_\epsilon,\gamma_\varphi u^h_\epsilon)  
\le \alpha_\varphi c_\varphi^2 \|u^h_\epsilon-v^h\|_{X_\varphi}^2
+\varphi(\gamma_\varphi v^h,\gamma_\varphi v^h)-\varphi(\gamma_\varphi v^h,\gamma_\varphi u^h_\epsilon). \]
Also, by \eqref{eq2.1}, we have two constants $\overline{c}$ and $\tilde{c}$, dependent on $v^h$ 
but independent of $u^h_\epsilon$ such that
\[ -\varphi(\gamma_\varphi v^h,\gamma_\varphi u^h_\epsilon)\le -\overline{c}-\tilde{c}\,\|u^h_\epsilon\|_X. \]
Hence, for any fixed $v^h\in X^h$, there is a constant $c(v^h)$, dependent on $v^h$ but independent of 
$\epsilon$, such that
\[ \frac{1}{\epsilon}\langle P u^h_\epsilon,u^h_\epsilon-v^h\rangle\le c(v^h). \]
Then, we deduce that
\begin{equation}
\limsup_{\epsilon\to 0}\langle P u^h_\epsilon,u^h_\epsilon-v^h\rangle \le 0\quad\forall\,v^h\in X^h.
\label{eq11a}
\end{equation}
By assumption $(A_9)$, for any $v\in X$, there exists $v^h\in X^h$ such that $v^h\to v$ in $X$.  Since 
$\{\|P u^h_\epsilon\|_{X^*}\}_{h,\epsilon>0}$ is uniformly bounded, we derive from \eqref{eq11a} that
\[ \limsup_{h,\epsilon\to 0}\langle P u^h_\epsilon,u^h_\epsilon-v\rangle \le 0\quad\forall\,v\in X. \]
Now that $P$ is pseudomonotone and $u^h_\epsilon\weak w$, implying
\[ \langle P w,w-v\rangle\le \liminf_{h,\epsilon\to 0}\langle P u^h_\epsilon,u^h_\epsilon-v\rangle 
 \quad\forall\,v\in X. \]
Combine the last two inequalities to get
\[ \langle P w,w-v\rangle\le 0\quad\forall\,v\in X. \]
From this relation, we conclude that
\[ \langle P w,v\rangle=0\quad\forall\,v\in X, \]
and hence, 
\[ w\in {\rm Ker}(P)=K. \]

Let us then prove that the weak limit $w$ is the solution of Problem (P).  Let $w^h\in K\cap X^h$ 
be such that
\[ \|w^h-w\|_X\to 0\quad {\rm as\ }h\to 0. \]
In \eqref{eq2}, we take $v^h=w^h$ to get
\begin{align*}
\langle Au^h_\epsilon, u^h_\epsilon-w^h \rangle 
&\le \frac{1}{\epsilon}\langle P u^h_\epsilon,w^h-u^h_\epsilon\rangle
+ \varphi(\gamma_\varphi u^h_\epsilon,\gamma_\varphi w^h)
- \varphi(\gamma_\varphi u^h_\epsilon,\gamma_\varphi u^h_\epsilon) \\
&\quad{} +j^0(\gamma_j u^h_\epsilon;\gamma_j w^h -\gamma_j u^h_\epsilon)-\langle f,w^h-u^h_\epsilon\rangle. 
\end{align*}
Since $Pw^h=0$ and $P$ is monotone,
\[ \frac{1}{\epsilon}\langle P u^h_\epsilon,w^h-u^h_\epsilon\rangle
 =-\frac{1}{\epsilon}\langle Pw^h-P u^h_\epsilon,w^h-u^h_\epsilon\rangle\le 0. \]
Hence,
\begin{align}
\langle Au^h_\epsilon, u^h_\epsilon-w^h \rangle 
&\le \varphi(\gamma_\varphi u^h_\epsilon,\gamma_\varphi w^h)
- \varphi(\gamma_\varphi u^h_\epsilon,\gamma_\varphi u^h_\epsilon) \nonumber \\
&\quad{} +j^0(\gamma_j u^h_\epsilon;\gamma_j w^h -\gamma_j u^h_\epsilon)-\langle f,w^h-u^h_\epsilon\rangle. 
\label{eq12}
\end{align}
Similar to \eqref{4.5a}, 
\begin{align*}
\varphi(\gamma_\varphi u^h_\epsilon,\gamma_\varphi w^h)
- \varphi(\gamma_\varphi u^h_\epsilon,\gamma_\varphi u^h_\epsilon)  
& \le \alpha_\varphi \|\gamma_\varphi (u^h_\epsilon-w)\|_{X_\varphi}
  \|\gamma_\varphi (u^h_\epsilon-w^h)\|_{X_\varphi}\\
&{}\quad +\varphi(\gamma_\varphi w,\gamma_\varphi w^h)-\varphi(\gamma_\varphi w,\gamma_\varphi u^h_\epsilon).
\end{align*}
For the terms on the right side of the above inequality, $\|\gamma_\varphi (u^h_\epsilon-w)\|_{X_\varphi}\to 0$, 
$\|\gamma_\varphi (u^h_\epsilon-w^h)\|_{X_\varphi}\to 0$ following the inequality
\[ \|\gamma_\varphi (u^h_\epsilon-w^h)\|_{X_\varphi}\le \|\gamma_\varphi (u^h_\epsilon-w)\|_{X_\varphi}
 +\|\gamma_\varphi (w-w^h)\|_{X_\varphi}, \]
and since $\varphi$ is continuous with respect to its second argument,
\[ \varphi(\gamma_\varphi w,\gamma_\varphi w^h)-\varphi(\gamma_\varphi w,\gamma_\varphi u^h_\epsilon)\to 0.\]
Thus,
\[\limsup_{h,\epsilon\to 0}\left[\varphi(\gamma_\varphi u^h_\epsilon,\gamma_\varphi w^h)
- \varphi(\gamma_\varphi u^h_\epsilon,\gamma_\varphi u^h_\epsilon) \right]\le 0. \]
Write
\begin{align*} 
j^0(\gamma_j u^h_\epsilon;\gamma_j w^h -\gamma_j u^h_\epsilon)
&= \left[j^0(\gamma_j u^h_\epsilon;\gamma_j w^h -\gamma_j u^h_\epsilon)
 +j^0(\gamma_j w^h;\gamma_j u^h_\epsilon -\gamma_j w^h)\right]\\
&{}\quad -j^0(\gamma_j w^h;\gamma_j u^h_\epsilon -\gamma_j w^h).
\end{align*}
Then,
\[ j^0(\gamma_j u^h_\epsilon;\gamma_j w^h -\gamma_j u^h_\epsilon)
\le \alpha_j \|\gamma_j (u^h_\epsilon -w^h)\|_{X_j}^2
+\left(c_0+c_1\|\gamma_j w^h\|_{X_j}\right) \|\gamma_j (u^h_\epsilon -w^h)\|_{X_j}, \]
and since $\|\gamma_j (u^h_\epsilon -w^h)\|_{X_j}\to 0$, we have
\[ \limsup_{h,\epsilon\to 0} j^0(\gamma_j u^h_\epsilon;\gamma_j w^h -\gamma_j u^h_\epsilon)\le 0. \]
Hence, from \eqref{eq12}, we derive 
\[ \limsup_{h,\epsilon\to 0} \langle Au^h_\epsilon, u^h_\epsilon-w^h \rangle \le 0. \]
This implies
\[ \limsup_{h,\epsilon\to 0} \langle Au^h_\epsilon, u^h_\epsilon-w \rangle \le 0. \]
By the pseudomonotonicity of $A$,
\begin{equation}
\langle Aw, w-v\rangle \le \liminf_{h,\epsilon\to 0} \langle Au^h_\epsilon, u^h_\epsilon-v \rangle
\quad\forall\,v\in X.
\label{eq13}
\end{equation}

Now fix an arbitrary element $v^{h^\prime}\in X^{h^\prime}\cap K$.  We take the upper limit as $h\to 0$ 
and $\epsilon\to 0$ along a subsequence of the spaces $X^h\supset X^{h^\prime}$ in \eqref{eq2} to obtain
\begin{align}
\limsup_{h,\epsilon\to 0} \langle Au^h_\epsilon, u^h_\epsilon-v^{h^\prime}\rangle 
&\le \varphi(\gamma_\varphi w,\gamma_\varphi v^{h^\prime})
- \varphi(\gamma_\varphi w,\gamma_\varphi w) \nonumber \\
&\quad{} +j^0(\gamma_j w;\gamma_j v^{h^\prime} -\gamma_j w)-\langle f,v^{h^\prime}-w\rangle. 
\label{eq14}
\end{align}
In the derivation of \eqref{eq14}, we used the inequality
\begin{align*}
\varphi(\gamma_\varphi u^h_\epsilon,\gamma_\varphi v^{h^\prime})
- \varphi(\gamma_\varphi u^h_\epsilon,\gamma_\varphi u^h_\epsilon)
&\le \alpha_\varphi\|\gamma_\varphi(u^h_\epsilon-w)\|_{X_\varphi}
\|\gamma_\varphi(v^{h^\prime}-u^h_\epsilon)\|_{X_\varphi} \\
&\quad {}+\varphi(\gamma_\varphi w,\gamma_\varphi v^{h^\prime})
- \varphi(\gamma_\varphi w,\gamma_\varphi u^h_\epsilon),
\end{align*}
the convergence $\|\gamma_\varphi(u^h_\epsilon-w)\|_{X_\varphi}\to 0$, the boundedness of
$\|\gamma_\varphi(v^{h^\prime}-u^h_\epsilon)\|_{X_\varphi}$, the continuity of $\varphi(\gamma_\varphi w,\cdot)$
on $X_\varphi$, and the upper continuity of $j^0(\cdot;\cdot)$ with respect to its two arguments.
Combine \eqref{eq13} and \eqref{eq14}, 
\[ \langle Aw, w-v^{h^\prime}\rangle \le \varphi(\gamma_\varphi w,\gamma_\varphi v^{h^\prime})
- \varphi(\gamma_\varphi w,\gamma_\varphi w)+j^0(\gamma_j w;\gamma_j v^{h^\prime} -\gamma_j w)
-\langle f,v^{h^\prime}-w\rangle.\]
Since $v^{h^\prime}\in X^{h^\prime}\cap K$ is arbitrary, we use the density of 
$\{X^{h^\prime}\cap K\}_{h^\prime}$ in $K$ to obtain
\begin{equation}
\langle Aw, w-v\rangle \le \varphi(\gamma_\varphi w,\gamma_\varphi v)
- \varphi(\gamma_\varphi w,\gamma_\varphi w)+j^0(\gamma_j w;\gamma_j v -\gamma_j w)
-\langle f,v-w\rangle\quad\forall\,v\in K.
\label{eq15}
\end{equation}
There, $w=u$ is the unique solution of Problem (P).

Since the limit $w=u$ is unique, we have the weak convergence of the entire family, i.e., 
\begin{equation}
u^h_\epsilon \weak u\quad {\rm in\ }X\ {\rm as\ } h,\epsilon\to 0. 
\label{eq16}
\end{equation}
Finally, let us prove the strong convergence.  We take $v^{h^\prime}=w^{h^\prime}$ in \eqref{eq14},
\begin{align*}
\limsup_{h,\epsilon\to 0} \langle Au^h_\epsilon, u^h_\epsilon-w^{h^\prime}\rangle 
&\le \varphi(\gamma_\varphi w,\gamma_\varphi w^{h^\prime})
- \varphi(\gamma_\varphi w,\gamma_\varphi w) \\
&\quad{} +j^0(\gamma_j w;\gamma_j w^{h^\prime} -\gamma_j w)-\langle f,w^{h^\prime}-w\rangle. 
\end{align*}
Then let $h^\prime\to 0$ and recall that $w=u$ to get
\begin{equation}
\limsup_{h,\epsilon\to 0} \langle Au^h_\epsilon, u^h_\epsilon-u \rangle\le 0.
\label{eq17}
\end{equation}
Apply \eqref{assum:A},
\[ m_A \|u^h_\epsilon -u\|_X^{2}\le \langle Au^h_\epsilon, u^h_\epsilon-u\rangle
-\langle Au, u^h_\epsilon-u\rangle.\]
By \eqref{eq17} and \eqref{eq16}, we conclude from the above inequality that
\[ m_A \|u^h_\epsilon -u\|_X^{2}\to 0\quad {\rm as\ }h,\epsilon\to 0, \]
i.e., we have the strong convergence $u^h_\epsilon \to u$ in $X$ as $h,\epsilon\to 0$. \hfill
\end{proof}

\section{Two relevant particular cases}\label{sec:particular}

In this section we consider two relevant particular cases of our results presented in  
the previous sections. They concern the case of a constrained hemivariational inequality 
(obtained when $\varphi\equiv 0$) as well as the case of a constrained variational inequality 
(obtained when $j\equiv 0$).

When $\varphi\equiv 0$ in Problem~(P), we have a pure hemivariational inequality from \eqref{hv}:  
\smallskip

{\sc Problem} $({\rm P})'$. 
{\it Find an element  $u \in K$ such that}
\begin{equation}\label{hv'}
\langle Au,v-u\rangle+ j^0(\gamma_j u;\gamma_j v -\gamma_j u)\ge\langle f,v-u\rangle\quad\forall\,v \in K.
\end{equation}

We need to modify $(A_7)$ and $(A_{10})$ for this particular case.

$(A_7)'$
\begin{equation}
\alpha_j c_j^2 < m_A. \label{small'}
\end{equation}

$(A_{10})'$ $\gamma_j\colon X\to X_j$ is compact.

Under the assumptions $(A_1)$, $(A_3)$, $(A_4)$, $(A_6)$, $(A_7)'$ and $(A_8)$, 
Problem $({\rm P})'$ has a unique solution. 

With the finite dimensional space $X^h$ and subset $K^h\subset X^h$ as at the beginning of 
Section \ref{s3}, we can introduce a penalty based numerical method for Problem $({\rm P})'$.

\smallskip
\noindent {\sc Problem} $({\rm P}^h_\epsilon)'$. 
{\it Find an element  $u^h_\epsilon\in X^h$ such that}
\begin{equation}
\langle Au^h_\epsilon,v^h-u^h_\epsilon\rangle+\frac{1}{\epsilon}\langle P u^h_\epsilon,v^h-u^h_\epsilon\rangle
 +j^0(\gamma_j u^h_\epsilon;\gamma_j v^h -\gamma_j u^h_\epsilon)\ge\langle f,v^h-u^h_\epsilon\rangle
\quad\forall\,v^h\in X^h. \label{eq2'}
\end{equation}

This problem has a unique solution under the assumptions $(A_1)$, $(A_3)$, $(A_4)$, $(A_6)$, 
$(A_7)'$ and $(A_8)$.

By Theorem \ref{thm:main}, we have the following convergence result for the penalty based numerical method.

\begin{corollary}\label{thm:main'}
Assume $(A_1)$, $(A_3)$, $(A_4)$, $(A_6)$, $(A_7)'$, $(A_8)$, $(A_9)$ and $(A_{10})'$.  Then,
\begin{equation}
 u^h_\epsilon\to u\quad{\rm in\ }X\ {\rm as\ }h,\epsilon\to 0. \label{4.9a'}
\end{equation}
\end{corollary}


When $j\equiv 0$ in Problem~(P), we have a variational inequality from \eqref{hv}:  
\smallskip

{\sc Problem} $({\rm P})''$. 
{\it Find an element  $u \in K$ such that}
\begin{equation}\label{hv''}
\langle Au,v-u\rangle+\varphi(\gamma_\varphi u,\gamma_\varphi v)-\varphi(\gamma_\varphi u,\gamma_\varphi u)
\ge\langle f,v-u\rangle\quad\forall\,v \in K.
\end{equation}

We modify $(A_7)$ and $(A_{10})$ for this particular case as follows.

$(A_7)''$
\begin{equation}
\alpha_\varphi c_\varphi^2 < m_A. \label{small''}
\end{equation}

$(A_{10})''$ $\gamma_\varphi\colon X\to X_\varphi$ is compact.

Under the assumptions $(A_1)$, $(A_2)$, $(A_4)$, $(A_5)$, $(A_7)''$ and $(A_8)$, 
Problem $({\rm P})''$ has a unique solution. 

With the finite dimensional space $X^h$ and subset $K^h\subset X^h$ as at the beginning of 
Section \ref{s3}, we can introduce a penalty based numerical method for Problem $({\rm P})''$.

\smallskip
\noindent {\sc Problem} $({\rm P}^h_\epsilon)''$. 
{\it Find an element  $u^h_\epsilon\in X^h$ such that}
\begin{equation}
\langle Au^h_\epsilon,v^h-u^h_\epsilon\rangle+\frac{1}{\epsilon}\langle P u^h_\epsilon,v^h-u^h_\epsilon\rangle
 +\varphi(\gamma_\varphi u^h_\epsilon,\gamma_\varphi v^h)
 -\varphi(\gamma_\varphi u^h_\epsilon,\gamma_\varphi u^h_\epsilon)\ge\langle f,v^h-u^h_\epsilon\rangle
\quad\forall\,v^h\in X^h. \label{eq2''}
\end{equation}

This problem has a unique solution under the assumptions $(A_1)$, $(A_2)$, $(A_4)$, $(A_5)$, $(A_7)''$ 
and $(A_8)$.

By Theorem \ref{thm:main}, we have the following convergence result for the penalty based numerical method.

\begin{corollary}\label{thm:main''}
Assume $(A_1)$, $(A_2)$, $(A_4)$, $(A_5)$, $(A_7)''$, $(A_8)$, $(A_9)$ and $(A_{10})''$.  Then,
\begin{equation}
 u^h_\epsilon\to u\quad{\rm in\ }X\ {\rm as\ }h,\epsilon\to 0. \label{4.9a''}
\end{equation}
\end{corollary}

\section{Applications in sample contact problems}\label{s4}

In this section, we illustrate applications of the convergence results established in the 
previous sections in the numerical solution of two static contact problems with constraints.  
The physical setting of a static contact problem, described with details in \cite{HS,MOS2013,SM2018} 
is as follows: the reference configuration of an elastic body is an open, bounded, 
connected set $\Omega\subset\real^d$ ($d=2$ or $3$ in applications) with a Lipschitz boundary 
$\Gamma=\partial\Omega$ partitioned into disjoint, measurable parts $\Gamma_1$, $\Gamma_2$ 
and $\Gamma_3$.  The body is in equilibrium under the action of a volume force of density 
$\fb_0$ in $\Omega$ and a surface traction of density $\fb_2$ on $\Gamma_2$; it is fixed on $\Gamma_1$
and is in contact on $\Gamma_3$ with one or two obstacles.   We assume ${\rm meas}\,(\Gamma_1)>0$.

For the description of the contact problems, we use the symbol $\mathbb{S}^{d}$ to denote the space of 
second order symmetric tensors on $\mathbb{R}^d$, and ``$\cdot$'' and ``$\|\cdot\|$'' will represent 
the canonical inner product and norm on the spaces $\mathbb{R}^d$ and $\mathbb{S}^d$.  We use
$\bu\colon \Omega\rightarrow\mathbb{R}^d$ for the displacement field and 
$\bsigma\colon \Omega\rightarrow\mathbb{S}^d$ for the stress field. Moreover, 
$\bvarepsilon(\bu):=\left(\nabla \bu+(\nabla \bu)^T\right)/2$ will represent the linearized strain 
tensor. Let $\bnu$ be the unit outward normal vector, which is defined a.e.\ on $\Gamma$.
For a vector field $\bv$, $v_\nu:=\bv\cdot\bnu$ and $\bv_\tau:=\bv-v_\nu\bnu$ are the normal 
and tangential components of $\bv$ on $\Gamma$. For the stress field $\bsigma$, 
$\sigma_\nu:=(\bsigma\bnu)\cdot\bnu$ and $\bsigma_\tau:=\bsigma\bnu-\sigma_\nu\bnu$ are
its normal and tangential components on the boundary.  

The two contact problems we consider in this section have the following equations and
boundary conditions in common:
\begin{align}
\label{H1} & \bsigma={\mathcal F}\bvarepsilon({\bu})\quad {\rm in}\ \Omega,\\
\label{H2} & {\rm Div}\,\bsigma+\fb_0=\bzero\quad{\rm in}\ \Omega,\\
\label{H3} & \bu=\bzero\quad{\rm on}\ \Gamma_1,\\
\label{H4} & \bsigma\bnu=\fb_2\quad{\rm on}\ \Gamma_2.
\end{align}
Equation \eqref{H1} is the elastic constitutive law where $\cF$ is the elasticity operator, 
\eqref{H2} represents the equilibrium equation, \eqref{H3} is the displacement boundary condition,  
and \eqref{H4} describes the traction boundary condition.  

We use the space
\[ V=\left\{\bv=(v_i)\in H^1(\Omega; \real^d) \mid \bv=\bzero\ {\rm a.e.\ on}\ \Gamma_1\right\}\]
or its subset for the displacement. Since ${\rm meas}\,(\Gamma_1)>0$, by Korn's inequality, 
$V$ is a Hilbert space with the inner product
\[ (\bu,\bv)_V:=\int_\Omega\bvarepsilon(\bu)\cdot\bvarepsilon(\bv)\,dx,\quad\bu,\bv\in V. \]
We denote the trace of a function $\bv\in H^1(\Omega; \real^d)$ on $\Gamma$ by the same symbol $\bv$.
We use the space $Q=L^2(\Omega; \mathbb{S}^d)$ for the stress and strain fields and we recall that $Q$ 
is a Hilbert space with the canonical inner product 
\[ (\bsigma,\btau)_Q:=\int_\Omega\sigma_{ij}(\bx)\,\tau_{ij}(\bx)\,dx,\quad \bsigma,\btau\in Q.\]

We assume that the elasticity operator $\cF\colon\Omega\times\mathbb{S}^d\to\mathbb{S}^d$ 
has the following properties.
\begin{equation}
\left\{\begin{array}{ll} 
{\rm (a)\  there\ exists}\ L_{\cal F}>0\ {\rm such\ that\ for\ all\ }\bvarepsilon_1,\bvarepsilon_2
    \in \mathbb{S}^d,\ {\rm a.e.}\ \bx\in \Omega,\\
   {}\qquad \|{\cal F}(\bx,\bvarepsilon_1)-{\cal F}(\bx,\bvarepsilon_2)\|
      \le L_{\cal F} \|\bvarepsilon_1-\bvarepsilon_2\|; \\
  {\rm (b)\  there\ exists}\ m_{\cal F}>0\ {\rm such\ that\ for\ all\ }\bvarepsilon_1,\bvarepsilon_2
    \in \mathbb{S}^d,\ {\rm a.e.}\ \bx\in \Omega,\\
    {}\qquad ({\cal F}(\bx,\bvarepsilon_1)-{\cal F}(\bx,\bvarepsilon_2))
       \cdot(\bvarepsilon_1-\bvarepsilon_2)\ge m_{\cal F}\,
      \|\bvarepsilon_1-\bvarepsilon_2\|^2;\\
{\rm (c) } \ {\cal F}(\cdot,\bvarepsilon)\ {\rm is\ measurable\ on\ }\Omega
\ {\rm for\ all \ }\bvarepsilon\in \mathbb{S}^d;\\
{\rm (d)}\ {\cal F}(\bx,\bzero)=\bzero\  {\rm for\ a.e.}\ \bx\in \Omega.
\end{array}\right.
\label{Fm}
\end{equation}
On the densities of the body force and the surface traction, we assume
\begin{equation}
\fb_0 \in L^2(\Omega;\real^d), \quad \fb_2\in L^2(\Gamma_2;\real^d).
\label{MOS30e}
\end{equation}
This regularity allows us to define the element
$\fb \in  V^*$ by equality
\begin{equation}
\langle\fb, \bv\rangle_{V^*\times V}=(\fb_0,\bv)_{L^2(\Omega;\real^d)}+(\fb_2,\bv)_{L^2(\Gamma_2;\real^d)}, 
\quad \bv\in V.
\label{MOS30f}
\end{equation}

We now complete the model \eqref{H1}--\eqref{H4} with specific contact conditions and friction laws.

\smallskip\noindent{\bf A unilateral frictional contact problem.}
In the first contact problem, we consider the case where the contact boundary $\Gamma_3$ consists of 
two disjoint measurable pieces, $\Gamma_{3,1}$ and $\Gamma_{3,2}$. On  $\Gamma_{3,1}$  the body is 
in contact with a perfectly rigid obstacle and we assume that the friction forces are negligible. 
Therefore, we model the contact with the frictionless Signorini unilateral contact condition, i.e., 
\begin{equation}
u_\nu\le0,\quad \sigma_\nu\le 0,\quad \sigma_\nu u_\nu=0,\quad \bsigma_\tau=\bzero\quad{\rm on}\ \ \Gamma_{3,1}.
\label{prob1.1}
\end{equation}
On $\Gamma_{3,2}$ the body is in persistent contact with a piston or a device, in such a way 
that the magnitude of the normal stress is limited by a given bound, denoted $F$. Moreover, 
when normal displacements occur, the reaction of the device is opposite to the displacement.
In addition, the contact is frictional and is modeled with a nonmonotone subdifferential boundary condition. 
These assumptions lead to the following boundary condition:
\begin{equation}
|\sigma_\nu|\le F,\quad -\sigma_\nu=F\,\frac{u_\nu}{\|u_\nu\|}\quad {\rm if}\ u_\nu\not=0,
\quad -\bsigma_\tau\in\partial j_\tau(\bu_\tau)\quad{\rm on}\ \Gamma_{3,2}.
\label{prob1.2}
\end{equation}
Here $\partial j_\tau$ is the Clarke subdifferential of a locally Lipschitz continuous
potential functional $j_\tau$.  

We assume that the bound $F$ and the potential function
$j_\tau\colon \Gamma_{3,2} \times \real^d \to\real$ 
have the following properties.
\begin{equation}\label{Fx}
F\in L^2(\Gamma_{3,2}),\qquad F(\bx)\ge 0\quad{\rm a.e.}\ \bx\in\Gamma_{3,2}.
\end{equation}
\begin{equation}
\label{jt}
\left\{
\begin{array}{l}
j_\tau \colon \Gamma_{3,2} \times\mathbb{R}^d \to \mathbb{R} \ \mbox{\rm is such that}\\[1mm]
\ \ {\rm (a)}\ j_\tau(\cdot, \bxi) \ \mbox{\rm is measurable on} \
\Gamma_{3,2} \ \mbox{\rm for all}\ \bxi \in \mathbb{R}^{d},\\[1mm]
\ \ {\rm (b)}\ j_\tau(\bx, \cdot) \ \mbox{\rm is locally Lipschitz on}
\ \real^d \ \mbox{\rm for a.e.}\ \bx \in \Gamma_{3,2},\\[1mm]
\ \ {\rm (c)}\ \| \partial j_\tau(\bx, \bxi) \|\le c_{0\tau} + c_{1\tau} \|\bxi\|\ \mbox{\rm for all}\
\bxi \in \mathbb{R}^d,\ \mbox{\rm a.e.}\ \bx \in \Gamma_{3,2} \\ [1mm]
\qquad \mbox{\rm with}\ c_{0\tau}, c_{1\tau} \ge 0,\\ [1mm]
\ \ {\rm (d)} \ j_\tau^0(\bx, \bxi_1; \bxi_2 - \bxi_1) + j_\nu^0(\bx, \bxi_2; \bxi_1 - \bxi_2) \le
\alpha_{j_\tau} \, \| \bxi_1 - \bxi_2 \|^2 \\ [1mm] 
\qquad \mbox{for a.e.} \ \bx \in \Gamma_{3,2}, \ \mbox{all} \ \bxi_1,\,
\bxi_2 \in \real^d \ \mbox{with} \ \alpha_{j_\tau} \ge 0.
\end{array}
\right. 
\end{equation} 

Then, the set of admissible displacement functions for the contact problem (\ref{H1})--(\ref{H4}), 
(\ref{prob1.1}), (\ref{prob1.2}) is 
\begin{equation}
U_0:=\left\{\bv\in V\mid v_\nu\le 0\ {\rm on\ }\Gamma_{3,1}\right\},
\label{prob1.3}
\end{equation}
and the weak formulation of the problem is the following.

{\sc Problem (P${}_1$)}.  \emph{Find a displacement field $\bu\in U_0$ such that}
\begin{align}
& \int_\Omega {\cal F}(\bvarepsilon({\bu}))\cdot\bvarepsilon(\bv-\bu)\,dx
+  \int_{\Gamma_{3,2}}F\left(|v_\nu|-|u_\nu|\right) ds \nonumber \\
&\qquad{} +\int_{\Gamma_{3,2}} j_\tau^0 (\bu_\tau; \bv_\tau-\bu_\tau)\, ds
\ge \langle\fb,\bv-\bu\rangle_{V^*\times V} \quad\forall\, \bv \in U_0. \label{prob1.4} 
\end{align}

Let $X=V$, $K=U_0$, $X_\varphi=L^2(\Gamma_{3,2})$ with $\gamma_\varphi$ the trace operator from
$V$ to $X_\varphi$, $X_j=L^2(\Gamma_{3,2})^d$ with $\gamma_j \bv=\bv_\tau$ for $\bv\in V$.  Define
\[ j(\gamma_j \bv)=\int_{\Gamma_{3,2}} j_\tau(\bv_\tau)\, ds,\quad \bv\in V.\]
Then, $\alpha_\varphi=0$ and $\alpha_j=\alpha_{j_\tau}$.  The smallness condition \eqref{small} 
takes the form
\begin{equation}
\alpha_{j_\tau}c_j^2< m_{\cal F},
\label{MOS30i1}
\end{equation}
where $c_j$ represents the norm of the trace operator $\gamma_j$.  Applying Theorem \ref{t1}, we know that 
under the stated assumptions and \eqref{MOS30i1}, there is a unique element $\bu\in U_0$ satisfying 
\begin{align}
& \int_\Omega {\cal F}(\bvarepsilon({\bu}))\cdot\bvarepsilon(\bv-\bu)\,dx
+  \int_{\Gamma_{3,2}}F\left(|v_\nu|-|u_\nu|\right) ds \nonumber \\
&\qquad{} +j^0 (\gamma_j\bu; \gamma_j\bv-\gamma_j\bu)
\ge \langle\fb,\bv-\bu\rangle_{V^*\times V} \quad\forall\, \bv \in U_0. \label{prob1.4aa} 
\end{align}
Since (\cite[Theorem 3.47]{MOS2013})
\[ j^0(\gamma_j\bu; \gamma_j\bv-\gamma_j\bu)\le\int_{\Gamma_{3,2}} j_\tau^0(\bu_\tau;\bv_\tau-\bu_\tau)\,ds,\]
$\bu\in U_0$ is also a solution of Problem (P${}_1$).  The uniqueness of a solution of Problem (P${}_1$)
can be verified directly by a standard approach.  Thus, under the stated assumptions and \eqref{MOS30i1}, 
Problem (P${}_1$) has a unique solution $\bu\in U_0$ .

Introduce an operator $P$ by
\begin{equation}
\langle P\bu,\bv\rangle=\int_{\Gamma_3}(u_\nu)_{+}v_\nu ds, \quad \bu,\bv\in V.
\label{ex:P1}
\end{equation}
Here and below, $r_+$ denotes the positive part of $r$.  It is easy to verify that 
$P\colon V\to V^*$ is a penalty operator for the set $U_0$.  Therefore, the penalized formulation of 
Problem (P${}_1$) is to find $\bu_\epsilon\in V$ such that
\begin{align}
& \int_\Omega {\cal F}(\bvarepsilon({\bu_\epsilon}))\cdot\bvarepsilon(\bv-\bu_\epsilon)\,dx
+\frac{1}{\epsilon}\int_{\Gamma_{3,2}}(u_{\epsilon,\nu})_{+}(v_\nu-u_{\epsilon,\nu})\, ds
+  \int_{\Gamma_{3,2}}F\left(|v_\nu|-|u_{\epsilon,\nu}|\right) ds\nonumber\\
&\qquad{} +\int_{\Gamma_{3,2}} j_\tau^0 (\bu_{\epsilon,\tau}; \bv_\tau-\bu_{\epsilon,\tau})\, ds
\ge \langle\fb,\bv-\bu_\epsilon\rangle_{V^*\times V} \quad\forall\, \bv \in V. \label{ex1c1}
\end{align}

Let us use the finite element method for the numerical solution of Problem (P${}_1$).  For brevity, 
assume $\Omega$ is a polygonal/polyhedral domain and express the each part of the boundary, 
where a different type of boundary condition is specified, as unions of closed flat components with 
disjoint interior.  Let $\{{\cal T}^h\}$ be a regular family of partitions of into triangles/tetrahedrons 
such that if the intersection of one side/face of an element with one closed flat
component has a relative positive measure, then the side/face lies entirely in that closed flat component.
Construct the linear element space corresponding to ${\cal T}^h$:
\[ V^h=\left\{\bv^h\in C(\overline{\Omega})^d \mid \bv^h|_T\in \mathbb{P}_1(T)^d,\ T\in {\cal T}^h,\
   \bv^h=\bzero\ {\rm on\ }\Gamma_1\right\}. \]
Then the penalty based numerical method for Problem (P${}_1$) is as follows.

\noindent 
{\sc Problem (P${}_{1,\epsilon}^h$)}.  \emph{Find a displacement field $\bu^h_\epsilon\in V^h$ such that}
\begin{align}
& \int_\Omega {\cal F}(\bvarepsilon({\bu^h_\epsilon}))\cdot\bvarepsilon(\bv^h-\bu^h_\epsilon)\,dx
+\frac{1}{\epsilon}\int_{\Gamma_{3,1}}(u^h_{\epsilon,\nu})_{+}(v^h_\nu-u^h_{\epsilon,\nu})\,ds
+\int_{\Gamma_{3,2}}F\left(|v_\nu^h|-|u_{\epsilon,\nu}^h|\right)\,ds\nonumber\\
&\qquad{} +\int_{\Gamma_{3,2}} j_\tau^0 (\bu_{\epsilon,\tau}^h; \bv_\tau^h-\bu_{\epsilon,\tau}^h)\,ds
\ge \langle\fb,\bv^h-\bu^h_\epsilon\rangle_{V^*\times V} \quad\forall\, \bv ^h\in V^h. \label{ex1a1}
\end{align}

The argument used in proving Theorem \ref{thm:main} is valid with $j^0(\cdot;\cdot)$ replaced 
by $\int_{\Gamma_{3,2}} j_\tau^0(\cdot;\cdot)\,ds$.  Thus, for the numerical solution 
$\bu^h_\epsilon$ of Problem (P${}_{1,\epsilon}^h$), we ascertain  the convergence:
\begin{equation}
 \bu^h_\epsilon\to \bu\quad{\rm in\ }V,\ {\rm as\ }h,\epsilon\to 0. \label{ex1b1}
\end{equation}
Indeed, it is routine to verify the conditions $(A_1)$--$(A_8)$ and $(A_{10})$ of Theorem \ref{thm:main} 
for Problem (P${}_1$) and Problem (P${}_{1,\epsilon}^h$).  Therefore, we restrict ourselves to 
examine the condition $(A_9)$.  For this, we note from \cite{HL77} and the explanations in 
\cite[Section 7.1]{HS} that $C^\infty(\overline{\Omega})^3\cap U_0$ is dense in $U_0$.  Thus, for 
any $\bv\in U_0$, we can first find a function $\tilde{\bv}\in C^\infty(\overline{\Omega})^3\cap U_0$ 
that is sufficiently close to $\bv$ in the norm of $V$; then by the finite element interpolation 
theory, we can approximate $\tilde{\bv}$ sufficiently closely by a finite element function 
$\bv^h\in V^h\cap U_0$ when the mesh-size $h$ is small enough.  Therefore, any function in $U_0$
can be approximated by a sequence of finite element functions that belong to $U_0$.

We note that in the special case where $\Gamma_{3,2}=\emptyset$ or $j_\tau$ vanishes,
Problem (P${}_1$) is simplified to a variational inequality.  The penalty based numerical method
Problem (P${}_{1,\epsilon}^h$) is similarly simplified and we have the convergence result \eqref{ex1b1} 
with simplified conditions, e.g., the condition \eqref{MOS30i1} is no longer needed.
Actually, in this case we are in a position to apply Corollary \ref{thm:main''}.

\smallskip\noindent{\bf A unilateral normal compliance frictional contact problem.} 
In the second contact problem, the boundary conditions on the contact surface are (cf.\ \cite{MOS30})
\begin{align}
& u_\nu\le g,\ \sigma_\nu+\xi_\nu\le 0,\ (u_\nu-g)\,(\sigma_\nu+\xi_\nu)=0,\ \xi_\nu\in\partial j_\nu(u_\nu)
\quad{\rm on\ }\Gamma_3, \label{MOS30a}\\[1mm]
& \|\bsigma_\tau\|\le F_b(u_\nu),\ -\bsigma_\tau=F_b(u_\nu)\,\frac{\bu_\tau}{\|\bu_\tau\|}
\ {\rm if\ }\bu_\tau\not=\bzero
\quad{\rm on\ }\Gamma_3. \label{MOS30b}
\end{align}

In condition \eqref{MOS30a}, inequality $u_\nu\le g$ restricts the allowed penetration and $j_\nu$ 
is a given potential function.  The contact condition \eqref{MOS30a} represents a combination of 
the Signorini contact condition for contact with a rigid foundation and the normal compliance condition 
for contact with a deformable foundation.  It models the contact with an obstacle made of a rigid body
covered with a soft layer of deformable material of thickness $g$.  Details and various mechanical 
interpretations can be found, e.g., in \cite{SM2018}.  The tangential contact condition \eqref{MOS30b}
represents a version of Coulomb's law of dry friction. Here  $F_b$ denotes the friction bound, 
assumed to  depend on the normal displacement $u_\nu$.  We now consider the following hypothesis
on the thickness $g:\Gamma_3\to\real$, the potential function $j_\nu\colon\Gamma_3\times\real\to\real$
and the friction bound $F_b\colon \Gamma_3\times\mathbb R\to \mathbb R_+$.
\begin{align}
&\label{gm}
\qquad\quad\ g\in L^2(\Gamma_3),\quad g(\bx)\ge 0\quad {\rm a.e.\ on}\ \Gamma_3.\\[2mm] 
&\left\{\begin{array}{ll} 
{\rm (a)} \ j_\nu (\cdot, r) \ \mbox{is measurable on} \ \Gamma_3 \ \mbox{for all} \ r \in \real
 \ \mbox{and there} \\
 \qquad  \mbox{exists} \ {\bar{e}} \in L^2(\Gamma_3)
\ \mbox{such that} \ j_\nu(\cdot, {\bar{e}} (\cdot)) \in L^1(\Gamma_3); \\ [1mm]
 {\rm (b)} \ j_\nu(\bx, \cdot)\ \mbox{is locally Lipschitz on}\ \real\ \mbox{for a.e.}\ \bx\in\Gamma_3;\\[1mm]
{\rm (c)} \ | \partial j_\nu(\bx, r) | \le {\bar{c}}_0 + {\bar{c}}_1|r|\ \mbox{for a.e.}\ \bx\in\Gamma_3\ 
 \forall\, r \in \real \ \mbox{with}\ {\bar{c}}_0, \, {\bar{c}}_1 \ge 0; \\[1mm]
{\rm (d)} \ j_\nu^0(\bx, r_1; r_2 - r_1) + j_\nu^0(\bx, r_2; r_1 - r_2) \le\alpha_{j_\nu}| r_1 - r_2 |^2\\
\qquad \mbox{for a.e.}\ \bx \in \Gamma_3, \ \mbox{all}\ r_1,r_2 \in \real\ \mbox{with}\ \alpha_{j_\nu}\ge 0.
\end{array}
\right.
\label{MOS30c}\\[2mm]  
&\left\{\begin{array}{ll}
{\rm (a)\ } \mbox{ There exists }L_{F_b}>0\mbox{ such that }\\
\qquad |F_b(\bx,r_1)-F_b(\bx,r_2)|\leq L_{F_b}|r_1-r_2|\quad\forall\,r_1,r_2\in{\mathbb R},
\ {\rm a.e.\ }\bx\in\Gamma_3;\\
{\rm (b)\ } F_b(\cdot,r){\rm\ is\ measurable\ on\ }\Gamma_3,\ {\rm for\ all\ }r\in \mathbb{R};\\
{\rm (c)}\ F_b(\bx,r)=0\ {\rm for}\ r\le 0,\ F_b(\bx,r)\ge 0\ {\rm for}\ r\ge 0,\ {\rm a.e.\ }\bx\in\Gamma_3.
\end{array}\right.
\label{MOS30d}
\end{align}
Then, the set of admissible displacement functions  for the contact problem (\ref{H1})--(\ref{H4}), 
(\ref{MOS30a}), (\ref{MOS30b}) is 
\[ U:=\left\{\bv\in V\mid v_\nu\le g\ {\rm on\ }\Gamma_3\right\}. \]
The weak formulation of this  problem is the following.
\smallskip

{\sc Problem (P${}_2$)}.  \emph{Find a displacement field $\bu\in U$ such that}
\begin{eqnarray}
&& \int_\Omega {\cal F}(\bvarepsilon({\bu}))\cdot\bvarepsilon(\bv-\bu)\,dx
+  \int_{\Gamma_3}F_b(u_\nu)\left(\|\bv_\tau\|-\|\bu_\tau\|\right) ds  \nonumber\\[1mm]
&&\qquad{}+\int_{\Gamma_3} j_\nu^0 (u_\nu; v_\nu-u_\nu)\, ds
\ge \langle\fb,\bv-\bu\rangle_{V^*\times V} \quad\forall\, \bv \in U. \label{MOS30h}
\end{eqnarray}

\smallskip
Let $X=V$, $K=U$, $X_\varphi=L^2(\Gamma_3)^d$ with $\gamma_\varphi$ the trace operator from $V$ to
$X_\varphi$, $X_j=L^2(\Gamma_3)$ with $\gamma_j \bv=v_\nu$ for $\bv\in V$.  Then, 
$\alpha_\varphi=L_{F_b}$ and $\alpha_j=\alpha_{j_\nu}$. Similar to the analysis of Problem (P${}_1$),
we can apply Theorem \ref{t1} and know that Problem (P${}_2$) has a unique solution $\bu\in U$ 
under the stated assumptions, and \eqref{small} takes the form
\begin{equation}
L_{F_b}\lambda_{1,V}^{-1}+\alpha_{j_\nu}\lambda_{1\nu,V}^{-1} < m_{\cal F},
\label{MOS30i}
\end{equation}
where $\lambda_{1,V}>0$ is the smallest eigenvalue of the eigenvalue problem
\[ \bu\in V,\quad \int_\Omega \bvarepsilon(\bu){\cdot}\bvarepsilon(\bv)\,dx=\lambda\int_{\Gamma_3}
 \bu{\cdot}\bv\,ds\quad\forall\,\bv\in V, \]
and $\lambda_{1\nu,V}>0$ is the smallest eigenvalue of the eigenvalue problem
\[ \bu\in V,\quad \int_\Omega \bvarepsilon(\bu){\cdot}\bvarepsilon(\bv)\,dx=\lambda\int_{\Gamma_3}
 u_\nu v_\nu ds\quad\forall\,\bv\in V.\]
 
Introduce an operator $P$ by
\begin{equation}
\langle P(\bu),\bv\rangle=\int_{\Gamma_3}(u_\nu-g)_{+}v_\nu ds, \quad \bu,\bv\in V.
\label{ex:P}
\end{equation}
It is easy to verify that $P\colon V\to V^*$ is a penalty operator for the set $U$.  Therefore, 
the penalized formulation of Problem (P${}_2$) consists to find $\bu_\epsilon\in V$ such that
\begin{eqnarray}
&& \int_\Omega {\cal F}(\bvarepsilon({\bu_\epsilon}))\cdot\bvarepsilon(\bv-\bu_\epsilon)\,dx
+\frac{1}{\epsilon}\int_{\Gamma_3}(u_{\epsilon,\nu}-g)_{+}(v_\nu-u_{\epsilon,\nu})\, ds
+  \int_{\Gamma_3}F_b(u_{\epsilon,\nu})\left(\|\bv_\tau\|-\|\bu_{\epsilon,\tau}\|\right)ds\nonumber\\[1mm]
&&\qquad{}+\int_{\Gamma_3} j_\nu^0 (u_{\epsilon,\nu}; v_\nu-u_{\epsilon,\nu})\, ds
\ge \langle\fb,\bv-\bu_\epsilon\rangle_{V^*\times V} \quad\forall\, \bv \in V. \label{ex1c}
\end{eqnarray}

We use the finite element setting already used in Problem (P${}_{1,\epsilon}^h$).
Then, the penalty based numerical method for Problem (P${}_2$) is as follows.

\medskip
\noindent {\sc Problem (P${}_{2,\epsilon}^h$)}.  
\emph{Find a displacement field $\bu^h_\epsilon\in V^h$ such that}
\begin{eqnarray}
&& \int_\Omega {\cal F}(\bvarepsilon({\bu^h_\epsilon}))\cdot\bvarepsilon(\bv^h-\bu^h_\epsilon)\,dx
+\frac{1}{\epsilon}\int_{\Gamma_3}(u^h_{\epsilon,\nu}-g)_{+}(v^h_\nu-u^h_{\epsilon,\nu})\, ds
+  \int_{\Gamma_3}F_b(u_\nu^h)\left(\|\bv_\tau^h\|-\|\bu_{\epsilon,\tau}^h\|\right) ds \nonumber\\[1mm]
&&\qquad{}+\int_{\Gamma_3} j_\nu^0 (u_{\epsilon,\nu}^h; v_\nu^h-u_{\epsilon,\nu}^h)\, ds
\ge \langle\fb,\bv^h-\bu^h_\epsilon\rangle_{V^*\times V} \quad\forall\, \bv ^h\in V^h. \label{ex1a}
\end{eqnarray}

Similar to the convergence discussion of the numerical method Problem (P${}_{1,\epsilon}^h$),  
again we need to examine the condition $(A_9)$:
\begin{equation}
\forall\,\bv\in U,\ \exists\,\bv^h\in V^h\cap U\ \ {\rm s.t.}\ \ \lim_{h\to 0}\|\bv^h-\bv\|_V=0. 
\label{property2}
\end{equation}
As is noted in \cite{HMS18}, if $C^\infty(\overline{\Omega})^d\cap U$ is dense in $U$ and 
the function $g$ is concave (in many applications, $g$ is constant), then \eqref{property2} is valid.
We assume this is the case.  Then we have the convergence of the numerical method defined by 
Problem (P${}_{2,\epsilon}^h$):
\begin{equation}
 \bu^h_\epsilon\to \bu\quad{\rm in\ }V,\ {\rm as\ }h,\epsilon\to 0. \label{ex1b}
\end{equation}

In the special case where $j_\nu$ is monotone, we have the convergence result for 
the penalty based numerical method of a constrained variational inequality.


\begin{thebibliography}{99}  

\bibitem{AH2009} 
K. Atkinson and W. Han, \emph{Theoretical Numerical Analysis: A Functional Analysis Framework}, 
third edition, Springer-Verlag, New York,  2009.







\bibitem{Clarke}
F. H. Clarke, \emph{Optimization and Nonsmooth Analysis},
Wiley, Interscience, New York, 1983.

\bibitem{DMP1}
Z. Denkowski, S. Mig\'orski and N.S. Papageorgiou,
{\it An Introduction to Non\-li\-near Analysis: Theory}, Kluwer
Academic/Plenum Publishers, Boston, Dordrecht, London, New York, 2003.


\bibitem{ET} I. Ekeland and R. Temam, \emph{Convex Analysis and Variational Problems}, 
North-Holland, Amsterdam, 1976.


\bibitem{GR1986} V. Girault and P.-A. Raviart, \emph{Finite Element Methods for
Navier-Stokes Equations, Theory and Algorithms}, Springer-Verlag, Berlin, 1986.

\bibitem{GLT1981}
R. Glowinski, J.-L. Lions, and R. Tr\'{e}moli\`{e}res, \emph{Numerical Analysis of Variational Inequalities}, 
North-Holland, Amsterdam, 1981.

\bibitem{Han17}
W. Han, Numerical analysis of stationary variational-hemivariational inequalities with applications
in contact mechanics, \emph{Mathematics and Mechanics of Solids} {\bf 23} (2018), 279--293.

\bibitem{HMS14}
W. Han, S. Mig\'orski and M. Sofonea, 
A class of variational-hemivariational inequalities with applications 
to frictional contact problems, \emph{SIAM Journal of Mathematical Analysis} 
{\bf 46} (2014), 3891--3912. 


\bibitem{HMS18}
W. Han, S. Mig\'orski and M. Sofonea, On a penalty based numerical method for unilateral contact problems
with non-monotone boundary condition, submitted.

\bibitem{HS}
W. Han and M. Sofonea, {\it Quasistatic  Contact Problems in
Viscoelasticity and Viscoplasticity}, Studies in Advanced
Mathematics {\bf 30}, Americal Mathematical Society, Providence,
RI--International Press, Somerville, MA, 2002.

\bibitem{HSB}
W. Han, M. Sofonea, and M. Barboteu,  Numerical analysis of elliptic hemivariational inequalities, 
\emph{SIAM J. Numer.\ Anal.} {\bf 55} (2017), 640--663.

\bibitem{HSD}
W. Han, M. Sofonea, and D. Danan, Numerical analysis of stationary variational-hemivariational 
inequalities, to appear in \emph{Numer.\ Math.}, DOI: 10.1007/s00211-018-0951-9. 

\bibitem{HMP1999}
J. Haslinger, M. Miettinen and P.D. Panagiotopoulos, {\em Finite
Element Method for Hemivariational Inequalities. Theory, Methods
and Applications}, Kluwer Academic Publishers, Boston, Dordrecht, London, 1999.

\bibitem{HL77}
I. Hlav\'a\v{c}ek and J. Lov\'{\i}\v{s}ek, A finite element analysis for the Signorini problem 
in plane elastostatics, \emph{Aplikace Matematiky} {\bf 22} (1977), 215--227.

\bibitem{KO1988}
N. Kikuchi and J.T. Oden, \emph{Contact Problems in Elasticity: A Study of 
Variational Inequalities and Finite Element Methods}, SIAM, Philadelphia, 1988.


\bibitem{MOS2013}
S. Mig\'orski, A. Ochal and M. Sofonea,
{\it Nonlinear Inclusions and Hemivariational Inequalities.
Models and Analysis of Contact Problems},
Advances in Mechanics and Mathematics \textbf{26}, Springer, New York, 2013.

\bibitem{MOS30}  
S. Mig\'orski, A. Ochal and M. Sofonea, 
A class of variational-hemivariational inequalities 
in reflexive Banach spaces, \emph{J. Elasticity} {\bf 127} (2017), 151--178.


\bibitem{NP1995}
Z. Naniewicz and P.D. Panagiotopoulos,
{\it Mathematical Theory of Hemivariational Inequalities
and Applications}, Marcel Dekker, Inc., New York, Basel, Hong Kong, 1995.

\bibitem{P1993}
P.D. Panagiotopoulos, \emph{Hemivariational Inequalities,
Applications in Mechanics and Engineering}, Springer-Verlag, Berlin, 1993.



\bibitem{SM2018}
M. Sofonea and S. Mig\'orski, \emph{Variational-Hemivariational Inequalities with Applications}, 
Chapman \& Hall/CRC Press, Boca Raton-London, 2018.	

\bibitem{SMH18}
M. Sofonea, S. Mig\'orski and W. Han, A penalty method for history-dependent variational-hemivariational 
inequalities, to appear in \emph{Computers and Mathematics with Applications}. 
	
\bibitem{SP14}
M. Sofonea and F. P\u atrulescu, Penalization of history-dependent variational inequalities,
\emph{European Journal of Applied Mathematics} {\bf 25} (2014), 155--176.

\bibitem{ZeidIIB} E. Zeidler, \emph{Nonlinear Functional Analysis and its
Applications. II/B: Nonlinear Monotone Operators}, Springer-Verlag, New
York, 1990.

\end{thebibliography}
\end{document}